\documentclass{amsart}
\usepackage{geometry}
\usepackage{graphicx}
\usepackage{amsrefs}
\usepackage[colorlinks]{hyperref}
\usepackage{mathpazo}

\newtheorem{theorem}{Theorem}[section]
\newtheorem{corollary}[theorem]{Corollary}

\theoremstyle{definition}
\newtheorem{definition}[theorem]{Definition}
\newtheorem{remark}[theorem]{Remark}
\newtheorem{example}[theorem]{Example}

\numberwithin{equation}{section}

\begin{document}
\title{Constructing topological biquandles via skew braces}
\author{Zhiyun Cheng}
\address{School of Mathematical Sciences, Beijing Normal University, Beijing 100875, China}
\email{czy@bnu.edu.cn}
\begin{abstract}
In this short note, we construct some non-trivial examples of topological biquandle. The key ingredient of the construction is the notion of skew brace.
\end{abstract}
\maketitle

\section{Introduction}\label{section1}
Quandle, which was independently introduced by Joyce \cite{Joy1982} and Matveev \cite{Mat1982} in 1982, is a set $Q$ equipped with a binary operation $\ast: Q\times Q\to Q$ satisfying several axioms motivated from Reidemeister moves in knot theory. For a given knot $K$ in $S^3$, similar to the knot group, one can associate a knot quandle $Q_K$ to $K$. It is known that two oriented knots have isomorphic knot quandles if and only if they are equivalent, or one can be obtained from the mirror image of the other one by reversing the orientation. Given a finite quandle $Q$, counting the homomorphisms from $Q_K$ to $Q$ gives rise to a coloring invariant Col$_Q(K)$ for a knot $K$. Analogous to the group homology, a homology theory for quandles was introduced in \cite{CJKLS2003}. By using quandle 2-cocycles, one can extend the coloring invariants Col$_Q(K)$ to the quandle cocycle invariants. As a generalization of the notion of quandle, biquandle, which is a set equipped with two binary operations, was introduced in \cite{KR2003} and explored in detail in \cite{FJK2004}. During the past twenty years, quandles and biquandles have been found to be closely related to Hopf algebras \cite{AG2003}, Frobenius algebras and Yang-Baxter equations \cite{CCEKS2008}, index type invariants \cite{Che2021} and quantum invariants \cite{BGPR2020}.

Similar to the notion of topological group, which is a topological space that is also a group such that the group operations are compatible with the topological structure, a topological quandle is a topological space with a quandle structure such that the binary operation is continues. Equipped with the compact-open topology, the set of all homomorphisms from a knot quandle to a fixed topological quandle also defines a knot invariant \cite{Rub2007}. This kind of knot invariants which take values in the set of topological spaces has been investigated deeply in recent years, see \cite{EM2016, CMS2018, ESZ2019}. A natural question is, for a given topological manifold, when does it admit a compatible non-trivial quandle structure? Surprisingly, recently it was proved by Tsvelikhovskiy that each topological manifold of positive dimension admits infinitely many non-trivial and non-isomorphic topological quandle structures \cite{Tsv2022}. It means that, unlike the group structures on topological manifolds, the quandle structures on topological manifolds are quite flexible and can be locally designed.

In this paper, we concern the construction of non-trivial topological biquandles. Here the \emph{non-triviality} means the topology is not the discrete topology and the biquandle is not a quandle. To the best knowledge of the author, at present we have no concrete non-trivial examples of topological biquandle. The main aim of this paper is to give some non-trivial examples of topological biquandles by using topological skew braces.

The outline of this paper is arranged as follows. In Section \ref{section2}, we take a brief review of the basics of biquandles and skew braces. The relation between skew braces and the set-theoretic solutions of the Yang-Baxter equation will also be discussed. In Section \ref{section3}, we give several concrete non-trivial examples of topological biquandles and discuss some applications of them.

\section{Background on biquandles and skew braces}\label{section2}
\subsection{Quandle and biquandle}
In this subsection, we give a quick introduction to the quandle theory.
\begin{definition}
A \emph{quandle} is a non-empty set $Q$ equipped with a binary operation $\ast: Q\times Q\to Q$ satisfying the following axioms:
\begin{enumerate}
  \item $\forall a\in Q, a\ast a=a$;
  \item $\forall b, c\in Q, \exists ! a\in Q$ such that $a\ast b=c$;
  \item $\forall a, b, c\in Q, (a\ast b)\ast c=(a\ast c)\ast(b\ast c)$.
\end{enumerate}
\end{definition}

These three quandle axioms correspond to the three Reidemeister moves in knot theory. As a consequence, if one suitably assigns an element of a finite quandle to each arc of a knot diagram, the number of quandle colorings does not depend on the choice of the knot diagram. Therefore, it gives rise to a knot invariant.

\begin{example}
Here we list some examples of quandle.
\begin{itemize}
  \item Any non-empty set $Q$ equipped with the binary operation $a\ast b=a$ $(\forall a, b\in Q)$ is called a \emph{trivial quandle}.
  \item Let $G$ be a group, for any $a, b\in G$ the binary operation $a\ast b=b^{-1}ab$ turns $G$ into a quandle, called the \emph{conjugation quandle} of $G$.
  \item Let $G$ be a group, for any $a, b\in G$ the binary operation $a\ast b=ba^{-1}b$ turns $G$ into a quandle, called the \emph{core quandle} of $G$.
  \item Let $\Sigma_g$ be a closed orientable surface of genus $g$ and $\mathcal{D}_g$ the set of isotopy classes of simple closed curves in $\Sigma_g$, for any $\alpha, \beta\in\mathcal{D}_g$, the binary operation $\alpha\ast\beta=T_{\beta}(\alpha)$ turns $\mathcal{D}_g$ into a quandle, called the \emph{Dehn quandle} of $\Sigma_g$ \cite{Zab1999}. Here $T_{\beta}$ denotes the Dehn twist along the curve $\beta$.
\end{itemize}
\end{example}

The notion of a quandle can be extended to that of a biquandle.

\begin{definition}
A \emph{biquandle} is a set $X$ with a bijective map $r: X\times X\to X\times X$, which sends $(a, b)$ to $r(a, b)=(b\star a, a\ast b)$ and satisfies the following axioms:
\begin{enumerate}
  \item The map $r$ satisfies the set-theoretic Yang-Baxter equation
  \begin{center}
  $(\text{Id}\times r)\circ(r\times\text{Id})\circ(\text{Id}\times r)=(r\times\text{Id})\circ(\text{Id}\times r)\circ(r\times\text{Id})$.
  \end{center} 
  \item For any $a\in X$, the maps $\star a: X\to X$ and $\ast a: X\to X$ are both bijective. These two maps induce a unique bijective map $S: X\times X\to X\times X$ such that 
  \begin{center}
  $S(b\star a, a)=(a\ast b, b)$.
  \end{center}
  \item For any $a\in X$, the map $S$ induces a bijection $\tau: X\to X$ on the diagonal 
  \begin{center}
  $S(a, a)=(\tau(a), \tau(a))$.
  \end{center}
\end{enumerate}
\end{definition}

For a given knot diagram, by associating each semi-arc with an element of a finite biquandle $X$ such that at each crossing point the four colors satisfy some conditions, one obtains a biquandle coloring invariant of knots. It is worthy to point out that, for classical knots in $S^3$, the knot biquandle contains exactly the same information as that of the knot quandle \cite{Ish2020}. More precisely, there exists a one-to-one correspondence between biquandle colorings and quandle colorings \cite{IT2024}. However, for some generalized knot theories, such as the virtual knot theory, invariants derived from biquandles contain more information comparing with that derived from quandles. 

\begin{example}
Here we list some examples of biquandles.
\begin{itemize}
  \item Let $(Q, \ast)$ be a quandle, for any $a, b\in Q$, the binary operation $a\star b=a$ turns $Q$ into a biquandle.
  \item Let $G$ be a group, for any $a, b\in G$, we define $a\ast b=b^{-1}a^{-1}b$ and $b\star a=a^2b$. Now $(G, \ast, \star)$ is a biquandle, which is called the \emph{Wada biquandle}.
  \item Let $X$ be a $\mathbb{Z}[t^{\pm1}, s^{\pm1}]$-module, the two binary operations $a\ast b=ta+(1-st)b$ and $b\star a=sb$ define a biquandle $(X, \ast, \star)$, which is called an \emph{Alexander biquandle}.
\end{itemize}
\end{example}

\subsection{Topological quandle and topological biquandle}
The concept of topological quandle was introduced by Rubinsztein in \cite{Rub2007}, which can be used to associate topological spaces to knots.

\begin{definition}
A \emph{topological quandle} is a topological space $Q$ equipped with a continuous binary operation $\ast: Q\times Q\to Q$ such that $(Q, \ast)$ is a quandle.
\end{definition}

Note that for a given topological quandle $Q$ and any $a\in Q$, the map $\ast a: Q\to Q$ defines an automorphism of $Q$. The reader is referred to \cite{EM2016} for some details of topological quandles. 

\begin{example}
Here we give some examples of topological quandles.
\begin{itemize}
  \item Any quandle equipped with the discrete topology is a topological quandle. On the other hand, any topological space equipped with the trivial quandle structure is a topological quandle.
  \item Any topological group $G$ equipped with the binary operation $a\ast b=b^{-1}ab$ is a topological quandle, called the \emph{conjugation quandle} of $G$.
  \item Let $M$ be a Riemannian manifold, and for any point $y\in M$ there exists an involution $i_y: M\to M$ such that $y$ is an isolated fixed point. Then $x\ast y=i_y(x)$ turns $M$ into a topological quandle. In particular, $S^n$ admits a topological quandle structure.
\end{itemize}
\end{example}

The notion of topological biquandle can be similarly defined.

\begin{definition}
A \emph{topological biquandle} is a topological space $X$ equipped with two continuous binary operations $\ast: X\times X\to X$ and $\star: X\times X\to X$ such that $(X, \ast, \star)$ is a biquandle.
\end{definition}

Obviously, assigning a biquandle the discrete topology turns it into a topological biquandle. On the other hand, any topological quandle $(X, \ast)$ is a topological biquandle if we define $b\star a=b$ for any $a, b\in X$.

\begin{remark}
In \cite{Hor2019}, Horvat also used the term topological biquandle to refer to a topologically defined biquandle associated to an oriented link, which turns out to be a quotient of the fundamental biquandle of the link. Hence the meaning of topological biquandle in \cite{Hor2019} is quite different from that used here.
\end{remark}

Let $X$ be a topological biquandle and $K$ an oriented knot which is realized as the closure of an $n$-braid $\beta$. Now $\beta$ induces a continuous map $f_{\beta}: X^n\to X^n$. Denote $J_X(K)\subset X^n$ to be the set of fixed points of $f_{\beta}$. It is not difficult to observe that $J_X(K)$ is nothing but the set of all $X$-colorings equipped with the compact-open topology. The following result follows directly from the definition of topological biquandle, while a similar result for topological quandle was first given in \cite[Theorem 4.1]{Rub2007}.

\begin{theorem}
For any topological biquandle $X$, the topological space $J_X(K)$ is a knot invariant.
\end{theorem}

\subsection{Skew brace}
Recall that a set-theoretic solution of the Yang-Baxter equation is a set $X$ with a bijective map $r: X\times X\to X\times X$ such that 
\begin{center}
$(\text{Id}\times r)\circ(r\times\text{Id})\circ(\text{Id}\times r)=(r\times\text{Id})\circ(\text{Id}\times r)\circ(r\times\text{Id})$.
\end{center} 
As before, let us denote $r(a, b)=(b\star a, a\ast b)$ for any $(a, b)\in X\times X$. If for any $a\in X$, the two maps $\ast a, \star a: X\to X$ are both bijective, then we say the solution $(X, r)$ is \emph{non-degenerate}. If $r^2=\text{id}$, then we say $(X, r)$ is \emph{involutive}. Notice that involutive solutions can only provide trivial knot invariants, since any knot invariant derived from an involutive solution is preserved under crossing change. However, it is still possible to use involutive solutions to define some non-trivial invariants for virtual knots \cite{CN2024}.

Braces were introduced by Rump in \cite{Rum2007} to study involutive, non-degenerate set-theoretic solutions of the Yang-Baxter equation. The notion of skew brace, which can be considered as a non-abelian version of brace, was introduced by Guarnieri and Vendramin in \cite{GV2017}.

\begin{definition}
A \emph{skew brace} is a triple $(A, +, \circ)$, where both $(A, +)$ and $(A, \circ)$ are groups, and for any $a, b, c\in A$ we have $a\circ(b+c)=a\circ b-a+a\circ c$.
\end{definition}

Usually, we call the group $(A, +)$ the \emph{additive group} of $A$ and call the group $(A, \circ)$ the \emph{multiplicative group} of $A$. Note that both of them are not necessary abelian groups. It is easy to find that the identity element of the additive group and that of the multiplicative group coincide, which is simply denoted by 0. For any $a\in A$, let us use $-a$ and $a'$ to denote the inverse of $a$ with respect to $+$ and $\circ$, respectively.

\begin{example}\label{example2.11}
Here we list some examples of skew braces.
\begin{itemize}
\item Any group $(A, +)$ equipped with the multiplication $a\circ b=a+b$ $(\forall a, b\in A)$ is a skew brace, which is called a \emph{trivial skew brace}.
\item Let $(R, +, \cdot)$ be a radical ring, i.e. $(R, \circ)$ is a group, here the Jacobson circle operation $\circ: R\times R\to R$ is defined as $a\circ b=a+a\cdot b+b$. Then $(R, +, \circ)$ is a skew brace.
\item Let $A$ and $X$ be two groups and $h: A\to \text{Aut}(X)$ be a group homomorphism. We introduce two binary operations on $X\times A$ as follows
\begin{center}
$(x, a)+(y, b)=(xy, ab)$ and $(x, a)\circ(y, b)=(xh_a(y), ab)$,
\end{center}
here $a, b\in A$, $h_a\in\text{Aut}(X)$ and $x, y\in X$. Then $(X\times A, +, \circ)$ is a skew brace, which is trivial if and only if $h$ is trivial.
\end{itemize}
\end{example}

One feature of skew braces is that they can be used to construct non-degenerate set-theoretic solutions of the Yang-Baxter equation. The following result was essentially proved in \cite[Theorem 3.1]{GV2017}, see also \cite[Theorem 4.1]{SV2018} or \cite[Theorem 2]{CN2024}.

\begin{theorem}\label{theorem2.12}
Let $(A, +, \circ)$ be a skew brace, then $(A, \ast, \star)$ is a biquandle, where $a\ast b=(-a+a\circ b)'\circ a\circ b$ and $b\star a=-a+a\circ b$.
\end{theorem}
\begin{proof}
We sketch the proof here.
\begin{enumerate}
\item The fact that the map $r(a, b)=(b\star a, a\ast b)$ satisfies the Yang-Baxter equation was proved in \cite[Theorem 3.1]{GV2017}. 
\item Assume $x\star a=b$, then $b=-a+a\circ x$, which follows that $b\star^{-1}a=x=a'\circ(a+b)$. On the other hand, if $y\ast a=b$, then $b=(-y+y\circ a)'\circ y\circ a$. It follows that $(-y+y\circ a)\circ b=y\circ a$, which implies that $y'\circ(-y+y\circ a)=a\circ b'$. Recall the formula $a\circ(-b+c)=a-a\circ b+a\circ c$, then we have $y'+a=a\circ b'$, which follows that $b\ast^{-1}a=y=(a\circ b'-a)'$. As a consequence, $\star a: A\to A$ and $\ast a: A\to A$ are both bijective.
\item The map $S: A\times A\to A\times A$ is given by $S(-a+a\circ b, a)=((-a+a\circ b)'\circ a\circ b, b)$. By using the formula $a\circ(b-c)=a\circ b-a\circ c+a$, it is not difficult to verify that $S(a, a)=(\tau(a), \tau(a))$, where $\tau(a)=-a'$.
\end{enumerate}
\end{proof}

\begin{remark}
If the biquandle $(A, \ast, \star)$ obtained from a skew brace $(A, +, \circ)$ is actually a quandle, then this skew brace is trivial and the quandle is the conjugation quandle of this group.
\end{remark}

\section{Some examples of topological biquandles}\label{section3}
\subsection{Biquandle structures on $\mathbb{R}^3$}
Consider the three dimensional Euclidean space $\mathbb{R}^3$, for any two points $(x_1, y_1, z_1)$ and $(x_2, y_2, z_2)$ in $\mathbb{R}^3$, we introduce the following two binary operations
\begin{center}
$(x_1, y_1, z_1)+(x_2, y_2, z_2)=(x_1+x_2, y_1+y_2, z_1+z_2)$,
\end{center}
and
\begin{center}
$(x_1, y_1, z_1)\circ(x_2, y_2, z_2)=(x_1+x_2, y_1+y_2, z_1+z_2+x_1y_2)$.
\end{center}
It is easy to find that both $(\mathbb{R}^3, +)$ and $(\mathbb{R}^3, \circ)$ are groups, where the former one is the ordinary additive group $\mathbb{R}^3$ and the latter one is actually the Heisenberg group $\mathbb{H}^1$. Note that now we have $(x_1, y_1, z_1)'=(-x_1, -y_1, x_1y_1-z_1)$.

\begin{theorem}\label{theorem3.1}
The triple $(\mathbb{R}^3, +, \circ)$ and the triple $(\mathbb{R}^3, \circ, +)$ are both skew braces.
\end{theorem}
\begin{proof}
Let us choose three points $a_i=(x_i, y_i, z_i)\in\mathbb{R}^3$ $(1\leq i\leq 3)$. 

First, we show $(\mathbb{R}^3, +, \circ)$ is a skew brace. One computes
\begin{flalign*}
a_1\circ(a_2+a_3)&=(x_1, y_1, z_1)\circ(x_2+x_3, y_2+y_3, z_2+z_3)\\
&=(\sum\limits_{i=1}^3x_i, \sum\limits_{i=1}^3y_i, \sum\limits_{i=1}^3z_i+x_1(y_2+y_3)).
\end{flalign*}
On the other hand,
\begin{flalign*}
a_1\circ a_2-a_1+a_1\circ a_3=&(\sum\limits_{i=1}^2x_i, \sum\limits_{i=1}^2y_i, \sum\limits_{i=1}^2z_i+x_1y_2)-(x_1, y_1, z_1)\\
&+(x_1+x_3, y_1+y_3, z_1+z_3+x_1y_3)\\
=&(\sum\limits_{i=1}^3x_i, \sum\limits_{i=1}^3y_i, \sum\limits_{i=1}^3z_i+x_1y_2+x_1y_3).
\end{flalign*}
We have $a_1\circ(a_2+a_3)=a_1\circ a_2-a_1+a_1\circ a_3$, hence $(\mathbb{R}^3, +, \circ)$ is a skew brace.

Second, we show $(\mathbb{R}^3, \circ, +)$ is also a skew brace. One computes
\begin{flalign*}
a_1+(a_2\circ a_3)=&(x_1, y_1, z_1)+(x_2+x_3, y_2+y_3, z_2+z_3+x_2y_3)\\
=&(\sum\limits_{i=1}^3x_i, \sum\limits_{i=1}^3y_i, \sum\limits_{i=1}^3z_i+x_2y_3),
\end{flalign*}
and
\begin{flalign*}
(a_1+a_2)\circ (a_1)'\circ(a_1+a_3)=&(x_1+x_2, y_1+y_2, z_1+z_2)\circ(-x_1, -y_1, x_1y_1-z_1)\\
&\circ(x_1+x_3, y_1+y_3, z_1+z_3)\\
=&(x_2, y_2, z_2-x_2y_1)\circ(x_1+x_3, y_1+y_3, z_1+z_3)\\
=&(\sum\limits_{i=1}^3x_i, \sum\limits_{i=1}^3y_i, \sum\limits_{i=1}^3z_i+x_2y_3).
\end{flalign*}
It follows that $a_1+(a_2\circ a_3)=(a_1+a_2)\circ (a_1)'\circ(a_1+a_3)$, therefore $(\mathbb{R}^3, \circ, +)$ is a skew brace.
\end{proof}

Combine Theorem \ref{theorem2.12} and Theorem \ref{theorem3.1} together, we obtain the following corollary.

\begin{corollary}\label{corollary3.2}
Both $(\mathbb{R}^3, r_1)$ and $(\mathbb{R}^3, r_2)$ are topological biquandles, where 
\begin{center}
$r_1((x_1, y_1, z_1), (x_2, y_2, z_2))=((x_2, y_2, z_2+x_1y_2), (x_1,y_1,z_1-x_2y_1))$
\end{center}
and 
\begin{center}
$r_2((x_1, y_1, z_1), (x_2, y_2, z_2))=((x_2, y_2, z_2-x_1y_2), (x_1, y_1, z_1+x_1y_2))$.
\end{center}
\end{corollary}

\begin{remark}
Let $(A, +, \circ)$ be a skew brace, then the solution $(A, \ast, \star)$ is involutive if and only if the additive group $(A, +)$ is abelian \cite{GV2017}. As a consequence, $(\mathbb{R}^3, r_1)$ is involutive and $(\mathbb{R}^3, r_2)$ is non-involutive.
\end{remark}

\begin{example}
Let $L=K_1\cup\cdots\cup K_n$ be an oriented $n$-component link and $X$ be the topological biquandle $(\mathbb{R}^3, r_2)$. Choose a link diagram of $L$ and two components $K_i$ and $K_j$. Let us use $C_{ij}\cup C_{ji}$ to denote the set of crossing points of $K_i$ with $K_j$, where for each crossing in $C_{ij}$ the lower strand belongs to $K_i$ and for each crossing in $C_{ji}$ the lower strand belongs to $K_j$. Now we define 
\begin{center}
$c_{ij}=\sum\limits_{c\in C_{ij}}w(c)$ and $c_{ji}=\sum\limits_{c\in C_{ji}}w(c)$,
\end{center}
where $w(c)$ denotes the sign of $c$. For simplicity, we set $c_{ii}=0$ for any $1\leq i\leq n$. Obviously, when $i\neq j$ we have $c_{ij}+c_{ji}=2lk(K_i, K_j)$. Now we assign a color $(x_i, y_i, z_i)\in\mathbb{R}^3$ to a point of $K_i$, then walk along $K_i$ according to the orientation. When we come back to the starting point, now the color is turned into $(x_i, y_i, z_i+\sum\limits_{j=1}^n(c_{ij}x_iy_j-c_{ji}x_jy_i))$. It turns out the coloring space 
\begin{center}
$J_X(L)=\{(x_1, y_1, z_1, \cdots, x_n, y_n, z_n)\in\mathbb{R}^{3n}|\sum\limits_{j=1}^nc_{ij}x_iy_j=\sum\limits_{j=1}^nc_{ji}x_jy_i, 1\leq i\leq n\}$.
\end{center}
Note that the equation corresponding to $K_n$ can be obtained from the rest $n-1$ equations, hence the equation $\sum\limits_{j=1}^nc_{nj}x_ny_j=\sum\limits_{j=1}^nc_{jn}x_jy_n$ can be removed.
\end{example}

\begin{remark}
Corollary \ref{corollary3.2} can be extended by using Heisenberg group $\mathbb{H}^n$. Recall that $\mathbb{H}^n$, as a topological space, is equal to $\mathbb{R}^n\times\mathbb{R}^n\times\mathbb{R}$. The binary operation $\circ: \mathbb{H}^n\times\mathbb{H}^n\to\mathbb{H}^n$ is defined as
\begin{center}
$(\alpha_1, \beta_1, z_1)\circ(\alpha_2, \beta_2, z_2)=(\alpha_1+\alpha_2, \beta_1+\beta_2, z_1+z_2+\alpha_1\cdot\beta_2)$,
\end{center}
where $(\alpha_i, \beta_i, z_i)\in\mathbb{R}^n\times\mathbb{R}^n\times\mathbb{R}$ $(i=1, 2)$ and $\alpha_1\cdot\beta_2$ denotes the dot product of $\alpha_1$ and $\beta_2$. As an analogy, we leave it to the reader to verify that both the two triples $(\mathbb{R}^{2n+1}, +, \circ)$ and $(\mathbb{R}^{2n+1}, \circ, +)$ are skew braces. As a consequence, we obtain two different topological biquandle structures on $\mathbb{R}^{2n+1}$.
\end{remark}

\subsection{Biquandle structures on $S^1\times\mathbb{R}^2$}
Consider the topological space $S^1\times\mathbb{R}^2$ equipped with the product topology. Choose two points $(e^{i\theta_1}, x_1, y_1)$ and $(e^{i\theta_2}, x_2, y_2)$ in $S^1\times\mathbb{R}^2$, the following two binary operations define two group structures on $S^1\times\mathbb{R}^2$:
\begin{center}
$(e^{i\theta_1}, x_1, y_1)+(e^{i\theta_2}, x_2, y_2)=(e^{i(\theta_1+\theta_2)}, x_1+x_2, y_1+y_2)$,
\end{center}
and
\begin{center}
$(e^{i\theta_1}, x_1, y_1)\circ(e^{i\theta_2}, x_2, y_2)=(e^{i(\theta_1+\theta_2)}, x_2\cos\theta_1-y_2\sin\theta_1+x_1, x_2\sin\theta_1+y_2\cos\theta_1+y_1)$.
\end{center}
The reader may have found that group $(S^1\times\mathbb{R}^2, \circ)$ is nothing but the Lie group $\text{SL}(2, \mathbb{R})$. In particular, we have $(e^{i\theta_1}, x_1, y_1)'=(e^{-i\theta_1}, -x_1\cos\theta_1-y_1\sin\theta_1, x_1\sin\theta_1-y_1\cos\theta_1)$.

\begin{theorem}\label{theorem3.6}
The triple $(S^1\times\mathbb{R}^2, +, \circ)$ and the triple $(S^1\times\mathbb{R}^2, \circ, +)$ are both skew braces.
\end{theorem}
\begin{proof}
We choose three points $a_k=(e^{i\theta_k}, x_k, y_k)\in S^1\times\mathbb{R}^2$, here $k\in\{1, 2, 3\}$.

First, we prove that $(S^1\times\mathbb{R}^2, +, \circ)$ is a skew brace. On one hand, we have
\begin{flalign*}
&a_1\circ(a_2+a_3)\\
=&(e^{i\theta_1}, x_1, y_1)\circ(e^{i(\theta_2+\theta_3)}, x_2+x_3, y_2+y_3)\\
=&(e^{i\sum\limits_{k=1}^3\theta_k}, (x_2+x_3)\cos\theta_1-(y_2+y_3)\sin\theta_1+x_1, (x_2+x_3)\sin\theta_1+(y_2+y_3)\cos\theta_1+y_1).
\end{flalign*}
On the other hand,
\begin{flalign*}
&a_1\circ a_2-a_1+a_1\circ a_3\\
=&(e^{i(\theta_1+\theta_2)}, x_2\cos\theta_1-y_2\sin\theta_1+x_1, x_2\sin\theta_1+y_2\cos\theta_1+y_1)-(e^{i\theta_1}, x_1, y_1)\\
&+(e^{i(\theta_1+\theta_3)}, x_3\cos\theta_1-y_3\sin\theta_1+x_1, x_3\sin\theta_1+y_3\cos\theta_1+y_1)\\
=&(e^{i\sum\limits_{k=1}^3\theta_k}, (x_2+x_3)\cos\theta_1-(y_2+y_3)\sin\theta_1+x_1, (x_2+x_3)\sin\theta_1+(y_2+y_3)\cos\theta_1+y_1).
\end{flalign*}
We conclude that $a_1\circ(a_2+a_3)=a_1\circ a_2-a_1+a_1\circ a_3$, hence $(S^1\times\mathbb{R}^2, +, \circ)$ is a skew brace.

Second, we show that $(S^1\times\mathbb{R}^2, \circ, +)$ is also a skew brace. One computes
\begin{flalign*}
&a_1+(a_2\circ a_3)\\
=&(e^{i\theta_1}, x_1, y_1)+(e^{i(\theta_2+\theta_3)}, x_3\cos\theta_2-y_3\sin\theta_2+x_2, x_3\sin\theta_2+y_3\cos\theta_2+y_2)\\
=&(e^{i\sum\limits_{k=1}^3\theta_k}, x_3\cos\theta_2-y_3\sin\theta_2+x_1+x_2, x_3\sin\theta_2+y_3\cos\theta_2+y_1+y_2).
\end{flalign*}
And
\begin{flalign*}
&(a_1+a_2)\circ a_1'\circ(a_1+a_3)\\
=&(e^{i(\theta_1+\theta_2)}, x_1+x_2, y_1+y_2)\circ(e^{-i\theta_1}, -x_1\cos\theta_1-y_1\sin\theta_1, x_1\sin\theta_1-y_1\cos\theta_1)\\
&\circ(e^{i(\theta_1+\theta_3)}, x_1+x_3, y_1+y_3)\\
=&(e^{i\theta_2}, (-x_1\cos\theta_1-y_1\sin\theta_1)\cos(\theta_1+\theta_2)-(x_1\sin\theta_1-y_1\cos\theta_1)\sin(\theta_1+\theta_2)+x_1+x_2),\\
&(-x_1\cos\theta_1-y_1\sin\theta_1)\sin(\theta_1+\theta_2)+(x_1\sin\theta_1-y_1\cos\theta_1)\cos(\theta_1+\theta_2)+y_1+y_2))\\
&\circ(e^{i(\theta_1+\theta_3)}, x_1+x_3, y_1+y_3)\\
=&(e^{i(\theta_1+\theta_2+\theta_3)},\\
&(x_1+x_3)\cos\theta_2-(y_1+y_3)\sin\theta_2+\\
&(-x_1\cos\theta_1-y_1\sin\theta_1)\cos(\theta_1+\theta_2)-(x_1\sin\theta_1-y_1\cos\theta_1)\sin(\theta_1+\theta_2)+x_1+x_2,\\
&(x_1+x_3)\sin\theta_2+(y_1+y_3)\cos\theta_2+\\
&(-x_1\cos\theta_1-y_1\sin\theta_1)\sin(\theta_1+\theta_2)+(x_1\sin\theta_1-y_1\cos\theta_1)\cos(\theta_1+\theta_2)+y_1+y_2))\\
=&(e^{i\sum\limits_{k=1}^3\theta_k}, x_3\cos\theta_2-y_3\sin\theta_2+x_1+x_2, x_3\sin\theta_2+y_3\cos\theta_2+y_1+y_2).
\end{flalign*}
It follows that $a_1+(a_2\circ a_3)=(a_1+a_2)\circ a_1'\circ(a_1+a_3)$, and $(S^1\times\mathbb{R}^2, \circ, +)$ is a skew brace.
\end{proof}

\begin{corollary}\label{corollary3.7}
Both $(S^1\times\mathbb{R}^2, r_1)$ and $(S^1\times\mathbb{R}^2, r_2)$ are topological biquandles, where 
\begin{center}
$r_1((e^{i\theta_1}, x_1, y_1), (e^{i\theta_2}, x_2, y_2))=((e^{i\theta_2}, x_2\cos\theta_1-y_2\sin\theta_1, x_2\sin\theta_1+y_2\cos\theta_1), (e^{i\theta_1}, x_1\cos\theta_2+y_1\sin\theta_2, -x_1\sin\theta_2+y_1\cos\theta_2))$
\end{center}
and 
\begin{center}
$r_2((e^{i\theta_1}, x_1, y_1), (e^{i\theta_2}, x_2, y_2))=((e^{i\theta_2}, x_2\cos\theta_1+y_2\sin\theta_1, -x_2\sin\theta_1+y_2\cos\theta_1), (e^{i\theta_1}, x_1+x_2-x_2\cos\theta_1-y_2\sin\theta_1, y_1+y_2+x_2\sin\theta_1-y_2\cos\theta_1))$.
\end{center}
\end{corollary}

\begin{remark}
Since $(S^1\times\mathbb{R}^2, +)$ is abelian and $(S^1\times\mathbb{R}^2, \circ)$ is non-abelian, as a solution to the Yang-Baxter equation, it follows that $(S^1\times\mathbb{R}^2, r_1)$ is involutive and $(S^1\times\mathbb{R}^2, r_2)$ is non-involutive.
\end{remark}

\begin{remark}
Consider the topological space $S^1\times\mathbb{R}^2$ as $S^1\times\mathbb{C}$, now for any two points $(e^{i\theta_1}, \alpha_1)$, $(e^{i\theta_2}, \alpha_2)\in S^1\times\mathbb{C}$, the two binary operations $+$ and $\circ$ can be rewritten as
\begin{center}
$(e^{i\theta_1}, \alpha_1)+(e^{i\theta_2}, \alpha_2)=(e^{i(\theta_1+\theta_2)}, \alpha_1+\alpha_2)$
\end{center}
and 
\begin{center}
$(e^{i\theta_1}, \alpha_1)\circ(e^{i\theta_2}, \alpha_2)=(e^{i(\theta_1+\theta_2)}, \alpha_1+e^{i\theta_1}\alpha_2)$.
\end{center}
In particular, we have $(e^{i\theta_1}, \alpha_1)'=(e^{-i\theta_1}, -e^{-i\theta_1}\alpha_1)$.
\end{remark}

\begin{example}
Consider the trefoil knot $K$, which can be realized as the closure of $\sigma_1^3\in B_2$. Let us use $Y$ to denote the topological biquandle $(S^1\times\mathbb{C}, r_2)$. Then the map induced by $\sigma_1^3$ sends the element $((e^{i\theta_1}, \alpha_1), (e^{i\theta_2}, \alpha_2))\in (S^1\times\mathbb{C})^2$ to $((e^{i\theta_2}, (e^{-i\theta_1}-e^{-i(\theta_1+\theta_2)})\alpha_1+(e^{-i\theta_1}-e^{-i(\theta_1+\theta_2)}+e^{-i(2\theta_1+\theta_2)})\alpha_2), (e^{i\theta_1}, (1-e^{-i\theta_1}+e^{-i(\theta_1+\theta_2)})\alpha_1+(1-e^{-i\theta_1}+e^{-i(\theta_1+\theta_2)}-e^{-i(2\theta_1+\theta_2)})\alpha_2))\in (S^1\times\mathbb{C})^2$. It follows that
\begin{center}
$J_Y(K)=\{((e^{i\theta_1}, \alpha_1), (e^{i\theta_2}, \alpha_2))\in (S^1\times\mathbb{C})^2|\theta_1=\theta_2, (1-e^{-i\theta_1}+e^{-2i\theta_1})(\alpha_1-e^{-i\theta_1}\alpha_2)=0\}$.
\end{center}
\end{example}

\subsection{Biquandle structures on $S^3\times S^1$}
We end this paper with an example of biquandle structures on a closed manifold. Let us consider the product manifold $S^3\times S^1=\operatorname{SU}(2)\times\operatorname{U}(1)$ equipped with the following two binary operations
\begin{center}
$(A_1, z_1)+(A_2, z_2)=(A_1A_2, z_1z_2)$
\end{center}
and
\begin{center}
$(A_1, z_1)\circ(A_2, z_2)=(A_1\begin{pmatrix}z_1 & 0\\0 & 1\end{pmatrix}A_2\begin{pmatrix}z_1^{-1} & 0\\0 & 1\end{pmatrix}, z_1z_2)$.
\end{center}
It is evident to see that both $(S^3\times S^1, +)$ and $(S^3\times S^1, \circ)$ are groups. In particular, the inverse operation of $\circ$ can be read as
\begin{center}
$(A_1, z_1)\circ^{-1}(A_2, z_2)=(A_1\begin{pmatrix}z_1z_2^{-1} & 0\\0 & 1\end{pmatrix}A_2^{-1}\begin{pmatrix}z_1^{-1}z_2 & 0\\0 & 1\end{pmatrix}, z_1z_2^{-1})$.
\end{center}

\begin{theorem}\label{theorem3.11}
The triple $(S^3\times S^1, +, \circ)$ and the triple $(S^3\times S^1, \circ, +)$ are both skew braces.
\end{theorem}
\begin{proof}
First we verify that
\begin{center}
$(A_1, z_1)\circ((A_2, z_2)+(A_3, z_3))=(A_1, z_1)\circ(A_2, z_2)-(A_1, z_1)+(A_1, z_1)\circ(A_3, z_3)$.
\end{center}
One computes
\begin{flalign*}
&(A_1, z_1)\circ((A_2, z_2)+(A_3, z_3))\\
=&(A_1, z_1)\circ(A_2A_3, z_2z_3)\\
=&(A_1\begin{pmatrix}z_1 & 0\\0 & 1\end{pmatrix}A_2A_3\begin{pmatrix}z_1^{-1} & 0\\0 & 1\end{pmatrix}, z_1z_2z_3).
\end{flalign*}
One the other hand,
\begin{flalign*}
&(A_1, z_1)\circ(A_2, z_2)-(A_1, z_1)+(A_1, z_1)\circ(A_3, z_3)\\
=&(A_1\begin{pmatrix}z_1 & 0\\0 & 1\end{pmatrix}A_2\begin{pmatrix}z_1^{-1} & 0\\0 & 1\end{pmatrix}, z_1z_2)-(A_1, z_1)+(A_1\begin{pmatrix}z_1 & 0\\0 & 1\end{pmatrix}A_3\begin{pmatrix}z_1^{-1} & 0\\0 & 1\end{pmatrix}, z_1z_3)\\
=&(A_1\begin{pmatrix}z_1 & 0\\0 & 1\end{pmatrix}A_2\begin{pmatrix}z_1^{-1} & 0\\0 & 1\end{pmatrix}A_1^{-1}, z_2)+(A_1\begin{pmatrix}z_1 & 0\\0 & 1\end{pmatrix}A_3\begin{pmatrix}z_1^{-1} & 0\\0 & 1\end{pmatrix}, z_1z_3)\\
=&(A_1\begin{pmatrix}z_1 & 0\\0 & 1\end{pmatrix}A_2A_3\begin{pmatrix}z_1^{-1} & 0\\0 & 1\end{pmatrix}, z_1z_2z_3).
\end{flalign*}
This finishes the proof of the statement that $(S^3\times S^1, +, \circ)$ is a skew brace. Next we turn to the proof of the statement that $(S^3\times S^1, \circ, +)$ is a skew brace, which needs to verify that 
\begin{center}
$(A_1, z_1)+(A_2, z_2)\circ(A_3, z_3)=((A_1, z_1)+(A_2, z_2))\circ^{-1}(A_1, z_1)\circ((A_1, z_1)+(A_3, z_3))$.
\end{center}
On one hand, we have
\begin{flalign*}
&(A_1, z_1)+(A_2, z_2)\circ(A_3, z_3)\\
=&(A_1, z_1)+(A_2\begin{pmatrix}z_2 & 0\\0 & 1\end{pmatrix}A_3\begin{pmatrix}z_2^{-1} & 0\\0 & 1\end{pmatrix}, z_2z_3)\\
=&(A_1A_2\begin{pmatrix}z_2 & 0\\0 & 1\end{pmatrix}A_3\begin{pmatrix}z_2^{-1} & 0\\0 & 1\end{pmatrix}, z_1z_2z_3).
\end{flalign*}
On the other hand,
\begin{flalign*}
&((A_1, z_1)+(A_2, z_2))\circ^{-1}(A_1, z_1)\circ((A_1, z_1)+(A_3, z_3))\\
=&(A_1A_2, z_1z_2)\circ^{-1}(A_1, z_1)\circ(A_1A_3, z_1z_3)\\
=&(A_1A_2\begin{pmatrix}z_2 & 0\\0 & 1\end{pmatrix}A_1^{-1}\begin{pmatrix}z_2^{-1} & 0\\0 & 1\end{pmatrix}, z_2)\circ(A_1A_3, z_1z_3)\\
=&(A_1A_2\begin{pmatrix}z_2 & 0\\0 & 1\end{pmatrix}A_1^{-1}\begin{pmatrix}z_2^{-1} & 0\\0 & 1\end{pmatrix}\begin{pmatrix}z_2 & 0\\0 & 1\end{pmatrix}A_1A_3\begin{pmatrix}z_2^{-1} & 0\\0 & 1\end{pmatrix}, z_1z_2z_3)\\
=&(A_1A_2\begin{pmatrix}z_2 & 0\\0 & 1\end{pmatrix}A_3\begin{pmatrix}z_2^{-1} & 0\\0 & 1\end{pmatrix}, z_1z_2z_3).
\end{flalign*}
The proof is finished.
\end{proof}

\begin{corollary}
Both $(S^3\times S^1, r_1)$ and $(S^3\times S^1, r_2)$ are topological biquandles, where 
\begin{center}
$r_1((A_1, z_1), (A_2, z_2))=((\begin{pmatrix}z_1 & 0\\0 & 1\end{pmatrix}A_2\begin{pmatrix}z_1^{-1} & 0\\0 & 1\end{pmatrix}, z_2), (\begin{pmatrix}z_1z_2^{-1} & 0\\0 & 1\end{pmatrix}A_2^{-1}\begin{pmatrix}z_1^{-1} & 0\\0 & 1\end{pmatrix}A_1\begin{pmatrix}z_1 & 0\\0 & 1\end{pmatrix}A_2\begin{pmatrix}z_1^{-1}z_2 & 0\\0 & 1\end{pmatrix}, z_1))$,
\end{center}
and
\begin{center}
$r_2((A_1, z_1), (A_2, z_2))=((\begin{pmatrix}z_1^{-1} & 0\\0 & 1\end{pmatrix}A_2\begin{pmatrix}z_1 & 0\\0 & 1\end{pmatrix}, z_2), (\begin{pmatrix}z_1^{-1} & 0\\0 & 1\end{pmatrix}A_2^{-1}\begin{pmatrix}z_1 & 0\\0 & 1\end{pmatrix}A_1A_2, z_1))$.
\end{center}
\end{corollary}

\begin{remark}
The result of Theorem \ref{theorem3.11} can be extended to $\operatorname{SU}(n)\times\operatorname{U}(1)$.
\end{remark}

\begin{remark}
We point out that actually all the biquandles $(\mathbb{R}^3=\mathbb{R}^2\times\mathbb{R}^1, r_1), (\mathbb{R}^2\times S^1, r_1)$ and $(S^3\times S^1, r_1)$ are in fact special instances of the third construction in Example \ref{example2.11}. Namely, one Lie group is expressed as the direct product of two Lie groups, while the other one corresponds to the semidirect product of the same pair of Lie groups. By contrast, it seems that the topological biquandles $(\mathbb{R}^3, r_2), (\mathbb{R}^2\times S^1, r_2)$ and $(S^3\times S^1, r_2)$ cannot be obtained in this way. For the three aforementioned examples, we observe that certain sets endowed with two distinct group structures form skew braces under either ordering of the two group operations. As a further illustration, consider the 4-element set equipped with two group structures: the cyclic group $\mathbb{Z}_4$ and the dihedral group $D_4$. One can readily verify that these two group structures yield two distinct 4-element skew braces. It is an interesting question to characterize the conditions under which two groups $(G, +)$ and $(G, \circ)$ over the same underlying set satisfy that both $(G, +, \circ)$ and $(G, \circ, +)$ are skew braces.
\end{remark}

\section*{Acknowledgements}
The author is grateful to Yang Su for his valuable suggestions. Zhiyun Cheng was supported by the NSFC grant 12371065.

\end{document}